\documentclass{amsart}[11pt]
\usepackage{amsmath, amsthm, amssymb, amscd, amsfonts,enumerate}
\usepackage[utf8]{inputenc}
\usepackage[mathscr]{eucal}
\usepackage{color}
\usepackage{float}
\usepackage{cite}
\usepackage[breaklinks]{hyperref}
\newtheorem{theorem}{THEOREM}[section]
\newtheorem{lemma}[theorem]{LEMMA}
\newtheorem{corollary}[theorem]{COROLLARY}

\usepackage{longtable}
\usepackage{threeparttablex}

\theoremstyle{definition}
\newtheorem{definition}{\emph{Definition}}
\newtheorem{exm}{\emph{Example}}

\theoremstyle{remark}

\newcommand{\integer}{\ensuremath{\mathbb{Z}}}
\newcommand{\rad}{\ensuremath{\textnormal{rad}}}
\newcommand{\nat}{\ensuremath{\mathbb{N}}}

\title{Coprimely Structured Modules}

\makeatletter
\renewcommand*{\@fnsymbol}[1]{\ifcase#1\or1\else\@arabic{\numexpr#1\relax}\fi}

\begin{document}

	\author{Zehra Bİlgİn}
	{\address{ \begin{flushright} \emph{ Department of Mathematics, Yıldız Technical University, 34210, Esenler,
		İstanbul, TURKEY}\\
\emph{ zbilgin@yildiz.edu.tr, zehrabilgin.zb@gmail.com}\end{flushright}}

\author{Kürşat Hakan Oral}
\address{  
	\begin{flushright} \emph{ 
				İstanbul, TURKEY}\\
			\emph{ khoral993@gmail.com}\end{flushright}}

	\date{\today}

	\keywords{Coprimely structured rings; multiplication modules; coprimely structured modules}
	\subjclass[2010]{13A15, 13C13, 13C99}

	\begin{abstract}
		Let $R$ be a commutative ring with identity. A prime submodule $P$ of an $R$-module $M$ is called coprimely structured if, whenever $P$ is coprime to each element of an arbitrary family of submodules of $M$, the intersection of the family is not contained in $P$. An $R$-module $M$ is called coprimely structured provided each prime submodule of $M$ is coprimely structured. In this paper, properties of coprimely structured modules are examined. Severals results for coprimely structured finitely generated modules and coprimely structured multiplication modules are obtained.
	\end{abstract}

	\maketitle
	
	\section{\bf{INTRODUCTION}}
	
		Throughout this paper, all rings are commutative and with identity. In \cite{AOT}, coprimely structured rings are studied. In this paper, we generalize this concept to the class of modules. Let $R$-be a ring. A prime submodule $Q$ of an $R$-module $M$ is called coprimely structured if for each family $\{N_i\}_{i\in I}$ of submodules of $R$ whenever $N_i+Q=M$ for each $i\in I$, we have $\bigcap_{i\in I}N_i\not\subseteq Q$. $M$ is called a copimely structured module if each of its prime submodules is coprimely structured. In Section 2, definitions and general results are given. Direct sums of modules are examined and a sufficient condition for a direct sum of coprimely structured modules being coprimely structured is stated [Theorem \ref{E:sum}].
	
	In Section 3, finitely generated modules are discussed. For a finitely generated module, it can be decided whether the module is coprimely structured or not by examining only maximal ideals instead of all prime ideals [Theorem \ref{E:max}]. The property of an $R$-module's being coprimely structured is transported to its localization in case of $R$ is local and the module is a finitely generated distributive $R$-module [Theorem \ref{E:local}].
	
	An $R$-module $M$ is called a multiplication module if each submodule $N$ of $M$ is of the form $IM$ for some ideal $I$ of $R$. Multiplication modules are studied widely in the literature, see \cite{BS, B}. In the category of multiplication modules, many properties of coprimely structured modules can be characterized. The radical of an ideal $I$ of $R$ is defined as the intersection of all prime ideals that contain $I$. Similarly, the radical of a submodule $N$, $\rad (N)$, of a module $M$ is the intersection of all prime submodules of $M$ that contains $N$. When $R$ is viewed as a module over itself, the definitions of the radical of an ideal and the radical of a submodule coincide. In \cite{RA}, the product of two submodules $N=IM$ and $K=JM$ of a multiplication $R$-module $M$ is defined as $(IJ)M$. Accordingly, the product of two elements $m,m'\in{M}$ is defined as the product of the submodules $Rm$ and $Rm'$. Using this definition, it is proved in \cite{RA} that $$\rad(N)=\{m\in{M} : m^k\subseteq{N} \textnormal{ for some } k\ge0\}$$ for a submodule $N$ of a multiplication $R$-module $M$. Section 4 is reserved for multiplication modules. A family $\{N_i\}_{i\in{I}}$ of submodules of a multiplication  $R$-module $M$ is said to satisfy property (*) if for each $x\in{M}$, there is an $n\in\nat$ such that $x\in\rad{(N_i)}$ implies $x^n\subseteq{N_i}$. With the aid of the property (*), it is possible to characterize coprimely structured multiplication modules in terms of prime submodules and maximal submodules [Theorem \ref{E:maximal}]. Besides, the property (*) is proved to be useful to give a sufficient condition for a module to be coprimely structured provided a particular quotient of the module is coprimely structured [Theorem \ref{E:zerorad}]. In case of $M$ is a finitely generated faithful multiplication $R$-module, $M$ is coprimely structured if and only if $R$ is coprimely structured [Theorem \ref{E:ringmod}].
	
	In Section 5, the relation between the property (*) and finitely generated multiplication modules is examined. In particular, if $R$ is a principal ideal ring and $M$ is a finitely generated faithful multiplication module, this property can be used to obtain information about $M$'s being zero-dimensional [Theorem \ref{E:zero-dim}].
	
	\section{\bfseries{COPRIMELY STRUCTURED MODULES}}
	In this section, we define coprimely structured modules and investigate some basic properties of them. Also we consider the relation between coprimely structured modules and strongly 0-dimensional modules.
	\begin{definition}
		Let $M$ be an $R$-module and $\{N_i\}_{i\in{I}}$ a family of submodules of $M$. A prime submodule $P$ of $M$ is said to be a coprimely structured submodule of $M$ if $N_i+P=M$, for all $i\in{I}$, implies $\bigcap_{i\in{I}}N_i\nsubseteq{P}$. An $R$-module $M$ is called a coprimely structured module if every prime submodule of $M$ is coprimely structured.
	\end{definition}
	Note that not every module is a coprimely structured module. Here is an example:
	\begin{exm}
		Let $R=\integer$, $M=\integer \times \integer$. Conside the family of submodules $\{N_{n}={n\integer \times n\integer} : n\in{\nat}, n \text{ odd } \} $ of $M$. We have
		$\bigcap_{n\in{\nat} \atop n \textnormal{odd}}N_{n} = (0)$. Observe that ${2\integer \times 2\integer}$ is a prime submodule of $M$ and $\bigcap_{n\in{\nat} \atop n \textnormal{odd}}N_{n}\subset{2\integer \times 2\integer}$. However, for each $n\in{\nat}$, $n$ odd, we have $N_n+2\integer \times 2\integer=M$. Thus $M$ is not coprimely structured.
	\end{exm}
	\begin{theorem}\label{E:homimage}
		Every homomorphic image of a coprimely structured module is coprimely structured.
	\end{theorem}
	\begin{proof}Let $M$ be a coprimely structured $R$-module, and $M'$ an $R$-module. Let $f:M\rightarrow{M'}$ be an $R$-module homomorphism. Assume that for a family of submodules $\{N_i'\}_{i\in{I}}$ and a prime submodule $P'$ of $f(M)$, the inclusion $\bigcap_{i\in{I}}N_i'\subseteq{P'}$ holds. Then there exist a family of submodules $\{N_i\}_{i\in{I}}$ of $M$ and a prime submodule $P$ of $M$ such that each $N_i$ and $P$ contain $\textnormal{ker}f$, and for all $i\in{I}$, the equalities $f(N_i)=N_i'$ and $f(P)=P'$ hold.
		Then
		$$f(\bigcap_{i\in{I}}N_i)\subseteq\bigcap_{i\in{I}}f(N_i)=\bigcap_{i\in{I}}N_i'\subseteq{P'}=f(P)$$
		and hence $\bigcap_{i\in{I}}N_i\subseteq{P}$. As $M$ is coprimely structured, there exists $j\in{I}$ such that $N_j+P\neq{M}$. Since both $N_j$ and $P$ contain $\textnormal{ker}f$, we conclude that $N_j'+P'=f(N_j+P)\neq{f(M)}$. Thus $f(M)$ is coprimely structured.
	\end{proof}
	\begin{corollary}\label{E:quotient}
		Let $M$ be an $R$-module and $N$ a submodule of $M$. If $M$ is coprimely structured, so is $M/N$.
	\end{corollary}
	A prime submodule $P$ of an $R$-module $M$ is called a strongly prime submodule provided for any family $\{N_i\}_{i\in{I}}$ of submodules of $M$, the inclusion $\bigcap_{i\in{I}}N_i\subseteq{P}$ implies $N_j\subseteq{P}$ for some $j\in{I}$. An $R$-module $M$ is a strongly $0$-dimensional module if each of its prime submodules is strongly prime. Strongly $0$-dimensional multiplication modules are introduced and studied in \cite{OAT}. The following theorem states the relation between coprimely structured modules and strongly $0$-dimensional modules.
	\begin{theorem}\label{E:0-dim}
		Every strongly $0$-dimensional $R$-module is coprimely structured.
	\end{theorem}
	\begin{proof}
		Let $M$ be a strongly $0$-dimensional $R$-module, $\{N_i\}_{i\in{I}}$ a family of submodules of $M$ and $P$ a prime submodule of $M$. Suppose that the equation $N_i+P=M$ is satisfied for all $i\in{I}$. Assume that $\bigcap_{i\in{I}}N_i\subseteq{P}$. Since $M$ is strongly $0$-dimensional, there exists $j\in{I}$ such that $N_j\subseteq{P}$. Then we have $M=N_j+P=P$ which is a contradiction. Therefore $\bigcap_{i\in{I}}N_i\not\subseteq{P}$.
	\end{proof}
	By Theorem \ref{E:homimage}, if a direct sum is coprimely structured so is each of its direct summands. We are to investigate when the converse is true. In \cite[2.1]{VE}, Erdoğdu characterizes the submodule structure of direct sum of two modules. For two $R$-modules $M$ and $N$, if $\text{Ann}(x)+\text{Ann}(y)=R$ for each $x\in M$ and $y\in N$, then every submodule of $M\oplus N$ is of the form $A\oplus B$ for some submodule $A$ of $M$ and some submodule $B$ of $N$. We generalize this result to an arbitrary direct sum.
	\begin{lemma}\label{E:ann}
		Let $M_i$, $i\in \ I$, be $R$-modules. The following are equivalent:
		\begin{enumerate}[(i)]
			\item $\text{Ann}(m_i)+\text{Ann}(m_j)=R$ for each  $m_i\in M_i$ and $i\neq j$.
			\item Each submodule of $\bigoplus_{i\in I}M_i$ is of the form $\bigoplus_{i\in I}N_i$, where $N_i$ is a submodule of $M_i$ for each $i\in I$.
		\end{enumerate}
	\end{lemma}
	
	\begin{proof}
		
		(i)$\Rightarrow$(ii) Suppose that $\text{Ann}(m_i)+\text{Ann}(m_j)=R$ for each  $m_i\in M_i$ and $i\neq j$. We first prove that $\text{Ann}((m_{i_1},...,m_{i_{n-1}}))+\text{Ann}(m_{i_n})=R$ for each $n\in\nat$, $m_{i_j}\in M_{i_j}$ and $1\leq j\leq n$. Let $n\in\nat$. We prove the statement by induction on $n$. Let $m_i\in M_i$. For $n=2$ the result follows from the assumption. Let $n=3$. By assumption, we have $1=a+b=c+d$ where $a\in \text{Ann}(m_1)$, $c\in \text{Ann}(m_2)$ and $b,d\in \text{Ann}(m_3)$ . Then
		$$1=(a+b)(c+d)=ac+ad+bc+bd.$$
		Since $ac\in \text{Ann}((m_1,m_2))$ and $ad+bc+bd\in \text{Ann}(m_3)$ we have $$\text{Ann}((m_1,m_2))+\text{Ann}(m_3)=R.$$
		Let $n=k$. Assume that $\text{Ann}((m_{i_1},...,m_{i_{l-1}}))+\text{Ann}(m_{i_l})=R$ for $2\leq l\leq k-1$, $i_j\in\{1,...,k\}$. Then,
		\begin{align}\text{Ann}((m_{i_1},...,m_{i_{k-2}},&m_{i_{k-1}}))+\text{Ann}(m_{i_k})\notag\\
		&=\text{Ann}(((m_{i_1},...,m_{i_{k-2}}),m_{i_{k-1}}))+\text{Ann}(m_{i_k})\notag\\
		&=R\notag
		\end{align}
		since, by assumption, we have $$\text{Ann}((m_{i_1},...,m_{i_{k-2}}))+\text{Ann}(m_{i_k})=R$$ and $\text{Ann}(m_{i_{k-1}})+\text{Ann}(m_{i_k})=R.$
		
		Now, let $N$ be a submodule of $M=\bigoplus_{i\in I}M_i$. Let $n\in N$. Then $n=\sum_{i\in I}m_i$, where $m_{i_1},...,m_{i_k}$ are nonzero and $m_i=0$ for $i\neq i_1,...,i_k$. For each $l\in\{1,...,k\}$, $$1=a_l+b_l$$ where $a_l\in\text{Ann}(((m_{i_1},...,m_{i_{l-1}},m_{i_{l+1}},...,m_{i_{k}})))$ and $b_l\in\text{Ann}(m_{i_l})$. Then
		\begin{align}1=\prod_{l=1}^k(a_l+b_l)=(a_1b_2...b_k&+a_2b_1b_3...b_k+...+a_kb_1...b_{k-1})\notag\\
		&+b_1...b_k+\left(\sum_{2\leq r} a_{t_1}...a_{t_r}b_{s_1}...b_{s_p}\right).\notag
		\end{align}
		Observe that the terms in the second line of the right hand side are contained in $\text{Ann}(m_{i_l})$ for each $l\in\{1,...,k\}$. On the other hand, for each $j\in\{1,...,k\}$, we have
		$$(a_jb_1...b_{j-1}b_{j+1}...b_k)n=(a_jb_1...b_{j-1}b_{j+1}...b_k)\iota_{i_j}(m_{i_j})\in N\cap\iota_{i_j}(M_{i_j})$$
		
		where $\iota_i:M_i\rightarrow\bigoplus_{i\in I}M_i$ is the $i$th natural injection. Hence
		$$n=1.n=\left[\prod_{l=1}^k(a_l+b_l)\right]n=\sum_{i\in I}c_i$$
		where $c_{i_1}=(a_1b_2...b_k)m_{i_1}$, $c_{i_2}=(a_2b_1b_3...b_k)m_{i_2}$,..., $c_{i_k}=(a_kb_1...b_{k-1})m_{i_k}$ and $c_i=0$ for $i\in I-\{i_1,...,i_k\}$.
		Then $n\in\bigoplus_{i\in I}(N\cap \iota_i(M_i))\subseteq N$.
		Therefore $N=\bigoplus_{i\in I}(N\cap \iota_i(M_i))$.
		
		(ii)$\Rightarrow$(i)Assume that $N$ is a submodule of $M=\bigoplus_{i\in I}M_i$. Then, by assumption, $N=\bigoplus_{i\in I}N_i$ for some submodule $N_i$ of $M_i$ for each $i\in I$. Observe that
		$$N\cap\iota_i(M_i)=\left(\bigoplus_{i\in I}N_i\right)\cap \iota_i(M_i)=\iota_i(N_i).$$
		Hence $N=\bigoplus_{i\in I}N_i=\bigoplus_{i\in I}(N\cap\iota_i(M_i))$.
		Let $i,j\in I$ and $m_i\in M_i$, $m_j\in M_j$. Set $a=\sum_{k\in I}a_k\in M$ where $a_i=m_i$, $a_j=m_j$ and $a_k=0$ for $k\neq i,j$. Since $Ra$ is a submodule of $M$, by the above argument $Ra=\bigoplus_{i\in I}(Ra\cap\iota_i(M_i))$. Then $a=\sum_{k\in I}n_k$ where $n_k\in Ra\cap\iota_k(M_k)$. Set $b=\sum_{k\in I}b_k\in M$ where $b_i=m_i$, $b_j=n_j$ and $b_k=0$ for $k\neq i,j$. Then $a-b=\sum_{k\in I}a_k-\sum_{k\in I}b_k=\sum_{k\in I}n_k-\sum_{k\in I}b_k$. Comparing the corresponding indices, we obtain $n_i-m_i=0$ and $m_j-n_j=0$. Hence $m_i=n_i$ and $m_j=n_j$. For each $k\in I$, since $\iota_k(n_k)\in Ra\cap\iota_k(M_k)$, there exists $r_k\in R$ such that $r_ka=\iota_k(n_k)\in \iota_k(M_k)$. In particular, $\iota_j(n_j)=r_ja.$ Using the equailities $n_i=m_i$ and $n_j=m_j$, we get $r_jm_i=0$ and $m_j=r_jm_j$. Then $r_j\in \text{Ann}(m_i)$ and $1-r_j\in\text{Ann}(m_j)$. Therefore
		$$1=r_j+(1-r_j)\in\text{Ann}(m_i)+\text{Ann}(m_j)$$ and hence $\text{Ann}(m_i)+\text{Ann}(m_j)=R$.
	\end{proof}
	\begin{theorem}\label{E:sum}
		Let $M_i$, $i\in I$, be coprimely structured $R$-modules and assume that $\text{Ann}(m_i)+\text{Ann}(m_j)=R$ for each,  $m_i\in M_i$, $i,j\in I$, $i\neq j$. Then $M=\bigoplus_{i\in I}M_i$ is coprimely structured.
	\end{theorem}
	\begin{proof} Let $N_\lambda$, $\lambda\in \Lambda$, be a family of submodules and $P$ a prime submodule of $M$ such that $\bigcap_{\lambda\in\Lambda}N_\lambda\subseteq P$.  By Lemma \ref{E:ann}, each submodule of $M$ is of the form $\bigoplus_{i\in I}N_i$ where $N_i$ is a submodule of $M_i$ for each $i \in I$. Then for each $\lambda\in\Lambda$, for each $i\in I$, there exists $N_{i,\lambda}$, submodule of $M_i$ such that $N_\lambda=\bigoplus_{i\in I}N_{i,\lambda}$ and there exists $P_i$, submodule of $M_i$, such that $P=\bigoplus_{i\in I}P_i$.
		Since $P$ is prime, there exists a unique $k\in I$ such that $P_k\neq M_k$. Let $r\in R$, $m\in M_k$. Assume that $rm\in P_k$ and $m\not\in P_k$. Set $a=\sum_{i\in I}a_i$ where $a_k=m$ and $a_i=0$ for $i\neq k$. Then $ra\in P$ and $a\not \in P$. Since $P$ is prime, $r\in (P:M)$. Then $rM_k\subseteq P_k$ and hence we conclude that $r\in (P_k:M_k)$. Therefore $P_k$ is a prime submodule of $M_k$.
		We have
		$$\bigoplus_{i\in I}\left(\bigcap_{\lambda\in\Lambda}N_{i,\lambda}\right)\subseteq \bigcap_{\lambda\in\Lambda}\left(\bigoplus_{i\in I}N_{i,\lambda}\right)=\bigcap_{\lambda\in\Lambda}N_\lambda\subseteq P=\bigoplus_{i\in I}P_i~~.$$
		Then, we have $\bigcap_{\lambda\in\Lambda}N_{k,\lambda}\subseteq P_k$. Since $M_k$ is coprimely structured, there exists $\gamma\in \Lambda$ such that $N_{k,\gamma}+P_k\neq M_k$. Therefore
		
		$$	N_\gamma+P=(N_{k,\gamma}+P_k)\oplus\left(\bigoplus_{i\in I \atop i\neq k}N_{i,\gamma}\right)+\left(\bigoplus_{i\in I \atop i\neq k}P_i\right)
		\neq M_k+\left(\bigoplus_{i\in I \atop i\neq k}M_i\right)=M~~.$$
		
		Thus, $M$ is coprimely structured.
	\end{proof}
	\begin{exm}
		Let $R=\mathbb{Z}$ and $M=\bigoplus_{p \text{ prime}}\mathbb{Z}_p$. Then $M$ is an $R$-module. For each prime number $p$, $\mathbb{Z}_p$, being finite, is coprimely structured. Let $p$ and $q$ be two different prime numbers. Since $p$ and $q$ are coprime, there exists $x$ and $y$ in $\mathbb{Z}$ such that $px+qy=1$. For each $m_p\in \mathbb{Z}_p$ and $m_q\in \mathbb{Z}_q$, since $px\in\text{Ann}(m_p)$ and $qy\in\text{Ann}(m_q)$ we have $1\in \text{Ann}(m_p)+\text{Ann}(m_q)$. Therefore, by Theorem \ref{E:sum}, $M$ is coprimely structured.
	\end{exm}
	\section{\bfseries{COPRIMELY STRUCTURED PROPERTY ON FINITELY GENERATED MODULES}}
	It is known that every proper submodule of a finitely generated $R$-module is contained in a maximal submodule, \cite[2.8]{FA}. Provided we work on the class of finitely generated modules, it is enough to consider maximal submodules to decide whether a module is coprimely structured, or not. The following theorem states this result.
	\begin{theorem}\label{E:max}
		Let $M$ be a finitely generated $R$-module. If every maximal submodule of $M$ is coprimely structured, then $M$ is coprimely structured.
	\end{theorem}
	\begin{proof}
		Assume that every maximal submodule of $M$ is coprimely structured.
		Let $\{N_i\}_{i\in{I}}$ be a family of submodules of $M$ and $P$ a prime submodule of $M$ satisfying $\bigcap_{i\in{I}}N_i\subseteq{P}$. Since $M$ is finitely generated, the submodule $P$ is contained in a maximal submodule $K$ of $M$. Then $\bigcap_{i\in{I}}N_i\subseteq{K}$, and since $K$ is coprimely structured, there exists $j\in{I}$ such that $N_j+K\neq{M}$. Then $N_j+P\neq{M}$. Thus, we conclude that $M$ is coprimely structured.
	\end{proof}
	\begin{lemma}\label{E:maxstrong}
		Let $M$ be a finitely generated $R$-module. The following are equivalent:
		\begin{enumerate}[(i)]
			\item $M$ is coprimely structured.
			\item Every maximal submodule $K$ of $M$ is strongly prime.
			\item  For any maximal submodule $K$ and any family $\{N_i\}_{i\in{I}}$ of submodules of $M$, $K+N_i=M$, for all $i\in{I}$, implies  $K+\bigcap_{i\in{I}}N_i=M$.
		\end{enumerate}
	\end{lemma}
	\begin{proof}
		(i)$\Rightarrow$(ii) Let $\{N_i\}_{i\in{I}}$ be a family of submodules of $M$ and $K$ a maximal submodule of $M$ satisfying $\bigcap_{i\in{I}}N_i\subseteq{K}$. Since $M$ is coprimely structured, $N_j+K\neq{M}$ for some $j\in{I}$. Then $N_j\subseteq{K}$ and hence $K$ is strongly prime.\\
		(ii)$\Rightarrow$(iii) Let $\{N_i\}_{i\in{I}}$ be a family of submodules of $M$ and $K$ a maximal submodule of $M$ such that $K+N_i=M$ holds for each $i\in{I}$ . We have $N_i\not\subseteq{K}$ for each $i\in{I}$. Since $K$ is strongly prime, we obtain $\bigcap_{i\in{I}}N_i\not\subseteq{K}$. This implies $K+\bigcap_{i\in{I}}N_i=M$.\\
		(iii)$\Rightarrow$(i): Let $\{N_i\}_{i\in{I}}$ be a family of submodules of $M$ and $P$ a prime submodule of $M$ satisfying $\bigcap_{i\in{I}}N_i\subseteq{P}$. Since $M$ is finitely generated, the submodule $P$ is contained in a maximal submodule $K$ of $M$. Then we have $\bigcap_{i\in{I}}N_i\subseteq{K}$, and hence $K+\bigcap_{i\in{I}}N_i\neq{M}$. This implies $P+N_j\subseteq{K+N_j}\neq{M}$ for some $j\in{I}$. Therefore $M$ is coprimely structured.
	\end{proof}
	It is proved in \cite[2.4]{OAT} that a strongly $0$-dimensional multiplication module is zero-dimensional. Actually, the proof is valid if we drop the assumption that the module is a multiplication module.
	\begin{theorem}\label{E:zero-dim}
		Let $M$ be a finitely generated $R$-module. Then $M$ is a zero-dimensional coprimely structured module if and only if $M$ is a strongly $0$-dimensional module.
	\end{theorem}
	\begin{proof} Follows from Theorem \ref{E:0-dim} and Lemma \ref{E:maxstrong}.
	\end{proof}
	An $R$-module $M$ is said to be a distributive module if the lattice of submodules of $M$ is distributive, that is, for any submodules $A,B,C$ of $M$, the equality $A\cap(B+C)=(A\cap B)+(A\cap C)$ holds. In \cite[2.4]{S}, Stephenson proved that for a local ring $R$ and a distributive $R$-module $M$, submodules of $M$ are comparable.  For a comprehensive study on distributive modules the reader may refer to \cite{VE, S}.
	\begin{theorem}\label{E:local}
		Let $R$ be a local ring and $M$ a finitely generated distributive module. Let $S$ be a multiplicatively closed subset of $R$. If $M$ is coprimely structured then $S^{-1}M$ is coprimely structured.
	\end{theorem}
	
	\begin{proof}
		Let $\{N_i\}_{i\in I}$ be a family of submodules and $P$ a prime submodule of $S^{-1}M$. Then for some family $\{K_i\}_{i\in I}$ of submodules and some prime submodule $Q$ of $M$ we have $N_i=S^{-1}K_i$ and $P=S^{-1}Q$. Assume that $\bigcap_{i\in I}N_i\subseteq P$. Then
		$$S^{-1}\left(\bigcap_{i\in I}K_i\right)\subseteq \bigcap_{i\in I}S^{-1}K_i=\bigcap_{i\in I}N_i\subseteq P=S^{-1}Q.$$
		Hence  $\bigcap_{i\in I}K_i\subseteq Q$. Since $M$ is coprimely structured, we have $K_j+Q\neq M$ for some $j\in I$. Since $M$ is finitely generated, there exists a maximal submodule $K$ of $M$ such that $K_j+Q\subseteq K$. As $K_j\subseteq K$, there exists a minimal prime submodule $Q_j$ of $M$ such that $K_j\subseteq Q_j \subseteq K$. Then $Q_j+Q\subseteq K$. Since $R$ is local and $M$ is distributive, either $Q_j\subseteq Q$ or $Q\subseteq Q_j$. Then $S^{-1}Q_j\subseteq S^{-1}Q$ or $S^{-1}Q \subseteq S^{-1}Q _j$. Hence we obtain $N_j+P\neq S^{-1}M$. Therefore $S^{-1}M$ is coprimely structured.
	\end{proof}

	\section{\bfseries{COPRIMELY STRUCTURED MULTIPLICATION MODULES}}
	In this section we study some properties of coprimely structured multiplication modules. An $R$-module $M$ is called a multiplication module if each submodule $N$ of $M$ is of the form $IM$ for some ideal $I$ of $R$
	As finitely generated modules, nonzero multiplication modules admits the property that every proper submodule is contained in a maximal submodule, by \cite [2.5] {BS}, we have the following theorems on multiplication modules similar to results on finitely generated modules mentioned above. The proofs are exactly the same, and hence omitted.
	\begin{theorem}
		Let $M$ be a multiplication $R$-module. If every maximal submodule of $M$ is coprimely structured, then $M$ is coprimely structured.
	\end{theorem}
	\begin{lemma}
		Let $M$ be a multiplication $R$-module. The following are equivalent:
		\begin{enumerate}[(i)]
			\item $M$ is coprimely structured.
			\item Every maximal submodule $K$ of $M$ is strongly prime.
			\item  For any maximal submodule $K$ and any family $\{N_i\}_{i\in{I}}$ of submodules of $M$, $K+N_i=M$, for all $i\in{I}$, implies  $K+\bigcap_{i\in{I}}N_i=M$.
		\end{enumerate}
	\end{lemma}
	\begin{theorem}
		Let $M$ be a zero-dimensional multiplication $R$-module. Then $M$ is coprimely structured if and only if $M$ is strongly $0$-dimensional.
	\end{theorem}
	
	Next, we prove a theorem that gives a characterization of coprimely structured multiplication modules in terms of families of prime submodules and maximal submodules. To this aim, we state some definitions and notations. For a submodule $N$ of $M$, the radical of $N$, denoted by $\rad(N)$, is defined as the intersection of all prime submodules of $M$ that contain $N$.  In \cite[3.3] {RA}, Ameri defines the product of two submodules $N=IM$ and $K=JM$ of a multiplication $R$-module $M$ as $(IJ)M$. Accordingly, the product of two elements $m,m'\in{M}$ is defined as the product of the submodules $Rm$ and $Rm'$. It is shown in \cite [3.13] {RA} that $\rad(N)=\{m\in{M} : m^k\subseteq{N} \textnormal{ for some } k\ge0\}$ for a submodule $N$ of a multiplication $R$-module $M$.
	
	A family $\{N_i\}_{i\in{I}}$ of submodules of a multiplication  $R$-module $M$ is said to satisfy property (*) if for each $x\in{M}$, there is an $n\in\nat$ such that $x\in\rad{(N_i)}$ implies $x^n\subseteq{N_i}$. We note that if we consider $R$ as a module over itself, this property is the same as the condition A2 in \cite[7]{A}. Accordingly, the following lemma is a generalization of \cite[2]{BR}. 
	\begin{lemma}\label{E:radical}
		A family $\{N_i\}_{i\in{I}}$ of submodules of a multiplication  $R$-module $M$ satisfies the property (*) if and only if for each subset $J\subseteq{I}$, $$\rad(\bigcap_{i\in{J}}N_i)=\bigcap_{i\in{J}}\rad(N_i).$$
	\end{lemma}
	\begin{proof}
		Let $\{N_i\}_{i\in{I}}$ be a family of submodules of a multiplication $R$-module $M$. Assume that the family $\{N_i\}_{i\in{I}}$ satisfies the property (*). Let $J$ be a subset of $I$. The inclusion $\rad(\bigcap_{i\in{J}}N_i)\subseteq\bigcap_{i\in{J}}\rad(N_i)$ always holds. For the reverse inclusion let $x\in{\bigcap_{i\in{J}}\rad(N_i)}$. Then for all $i\in{J}$ we have $x\in\rad(N_i)$. Since $\{N_i\}_{i\in{I}}$ satisfies the property (*), there exists an $n\in{\nat}$ such that $x^n\subseteq{N_i}$ for each $i\in{J}$. This implies $x^n\subseteq\bigcap_{i\in{J}}N_i$. Hence we obtain $x\in\rad(\bigcap_{i\in{J}}N_i)$. Conversely, assume for each subset $J$ of $I$, that the equation $\rad(\bigcap_{i\in{J}}N_i)=\bigcap_{i\in{J}}\rad(N_i)$ holds. Let $x\in{M}$ and set $J=\{i\in{I} : x\in{\rad(N_i)}\}$. Then $x\in\bigcap_{i\in{J}}\rad(N_i)=\rad(\bigcap_{i\in{J}}N_i)$. Therefore there is an $n\in{\nat}$ such that $x^n\in{\bigcap_{i\in{J}}N_i}$. Since $\bigcap_{i\in{J}}N_i\subseteq{N_i}$ for each $i\in{J}$, we conclude that $x^n\subseteq{N_i}$ for each $i\in{J}$. Therefore $\{N_i\}_{i\in{I}}$ satisfies the property (*).
		
	\end{proof}
	\begin{theorem}\label{E:maximal}
		Let $M$ be a multiplication $R$-module. If $M$ is coprimely structured, then for any family $\{P_i\}_{i\in{I}}$ of prime submodules and any maximal submodule $K$ of $M$, the inclusion $\bigcap_{i\in{I}}P_i\subseteq{K}$ implies  $P_j\subseteq{K}$ for some $j\in{I}$. The converse is true if the property (*) is satisfied by any family of submodules of $M$.
	\end{theorem}
	\begin{proof}
		Assume that $M$ is coprimely structured. Let $\{P_i\}_{i\in{I}}$ be a family of submodules of $M$ and $K$ a maximal submodule of $M$ such that $\bigcap_{i\in{I}}P_i\subseteq{K}$. We have $P_j+K\neq{M}$ for some $j\in{I}$. Thus $P_j\subseteq{K}$. Conversely, assume that the property (*) is satisfied by any family of submodules of $M$. Further, assume, for any family $\{P_i\}_{i\in{I}}$ of prime submodules and  any maximal submodule $K$ of $M$, that the inclusion  $\bigcap_{i\in{I}}P_i\subseteq{K}$ implies  $P_j\subseteq{K}$ for some $j\in{I}$. Let $\{N_\alpha\}_{\alpha\in{A}}$ be a family of submodules of $M$ and $P$ a prime submodule of $M$ such that $\bigcap_{\alpha\in{A}}N_\alpha\subseteq{P}$. Then $\rad(\bigcap_{\alpha\in{A}}N_\alpha)\subseteq\rad(P)=P$. As $M$ is a multiplication module, $P$ is contained in a maximal submodule $L$ of $M$. Since, by assumption, the property (*) is satisfied by $\{N_\alpha\}_{\alpha\in{A}}$, using Lemma \ref{E:radical}, we obtain $\bigcap_{\alpha\in{A}}\rad(N_\alpha)=\rad(\bigcap_{\alpha\in{A}}N_\alpha)\subseteq{P}\subseteq{L}$. Besides, for each $\alpha\in{A}$, $$\rad(N_\alpha)=\bigcap_{\beta\in{B} \atop N_\alpha\subseteq{P_{\beta,\alpha}}}P_{\beta,\alpha}$$ for some family $\{P_{\beta,\alpha}\}_{\beta\in{B}}$ of prime submodules of $M$. Therefore,
		$$\bigcap_{(\alpha,\beta)\in{A\times{B}} \atop N_\alpha\subseteq{P_{\beta,\alpha}}}P_{\beta,\alpha}=\bigcap_{\alpha\in{A}}\bigcap_{\beta\in{B} \atop N_\alpha\subseteq{P_{\beta,\alpha}}}P_{\beta,\alpha}=\bigcap_{\alpha\in{A}}\rad(N_\alpha)\subseteq{L}.$$
		Then, by assumption, we have $P_{\lambda,\kappa}\subseteq{L}$ for some $(\lambda,\kappa)\in{A\times{B}}$. This implies  $P_{\lambda,\kappa}+L\neq{M}$ for some $(\lambda,\kappa)\in{A\times{B}}$. Hence, for some $\lambda\in{A}$, we have $N_\lambda+P\subseteq{P_{\lambda,\kappa}}+P\subseteq{P_{\lambda,\kappa}}+L\neq{M}$. Thus $M$ is coprimely structured.
	\end{proof}
	\begin{theorem}\label{E:zerorad}
		Let $M$ be a multiplication $R$-module. Assume that the property (*) is satisfied for any family of submodules of $M$. Let $N$ be a submodule of $M$ which is contained in $\rad(0)$. Then $M/N$ is coprimely structured if and only if $M$ is coprimely structured.
	\end{theorem}
	\begin{proof}
		Assume that $M/N$ is coprimely structured. Let $\{P_i\}_{i\in{I}}$ be a family of prime submodules of $M$ and $K$ a maximal submodule of $M$ satisfying $\bigcap_{i\in{I}}P_i\subseteq{K}$. Then, as $N\subseteq{\rad(0)}$, we obtain $$\bigcap_{i\in{I}}P_i/N=(\bigcap_{i\in{I}}P_i)/N\subseteq{K/N}.$$ Since $K$ is maximal, $K/N$ is maximal in $M/N$. Then, for some $j\in{I}$, we have $P_j/N+K/N\neq{M/N}$. This implies $P_j/N\subseteq{K/N}$. Therefore, for some $j\in{I}$ the inclusion $P_j\subseteq{K}$ holds. Using Theorem \ref{E:maximal} we conclude that $M$ is coprimely structured. The converse follows from Corollary \ref{E:quotient}.
	\end{proof}
	\begin{theorem}\label{E:ringmod}
		Let $M$ be a finitely generated faithful multiplication $R$-module. $M$ is coprimely structured if and only if $R$ is coprimely structured.
	\end{theorem}
	\begin{proof}
		Assume that $M$ is a coprimely structured module. Let $\{I_\alpha\}_{\alpha\in{A}}$ be a family of ideals of $R$ and $P$  a prime ideal of $R$ satisfying $\bigcap_{\alpha\in{A}}I_\alpha\subseteq{P}$. Then, we have $\bigcap_{\alpha\in{A}}(I_\alpha{M})=\left(\bigcap_{\alpha\in{A}}I_\alpha\right)M\subseteq{PM}$. Since $M$ is coprimely structured, there exists $\beta\in{A}$ such that $(I_\beta+P)M=I_\beta{M}+PM\neq{M}$. Therefore, by \cite[3.1]{BS}, we obtain $I_\beta+P\neq{R}$, and hence $R$ is coprimely structured. Conversely, assume that $R$ is coprimely structured. Let $\{N_\lambda\}_{\lambda\in{L}}$ be a family of submodules of $M$ and $Q'$ a prime submodule of $M$. Suppose $\bigcap_{\lambda\in{L}}N_\lambda\subseteq{Q'}$. Since $M$ is a multiplication module there exist a family $\{I_\lambda\}_{\lambda\in{L}}$ of ideals of $R$ and a prime ideal $Q$ of $R$ such that $N_\lambda=I_\lambda{M}$, for all $\lambda\in{L}$, and $Q'=QM$, by \cite[2.11]{BS}. Then, using \cite[1.6]{BS}, we have $(\bigcap_{\lambda\in{L}}I_\lambda)M=\bigcap_{\lambda\in{L}}(I_\lambda{M})=\bigcap_{\lambda\in{L}}N_\lambda\subseteq{Q'}=QM$, and by \cite[3.1]{BS}, we obtain $\bigcap_{\lambda\in{L}}I_\lambda\subseteq{Q}$. Since $R$ is coprimely structured, there exists $\kappa\in{L}$ such that $I_\kappa+Q\neq{R}$. Then, using \cite[3.1]{BS}, we conclude that $N_\kappa+Q'=(I_\kappa{M}+P)M\neq{M}$. Thus $M$ is coprimely structured.
	\end{proof}
	
	\begin{theorem}
		Every Artinian multiplication module is coprimely structured.
	\end{theorem}
	\begin{proof}
		Every Artinian multiplication module is strongly $0$-dimensional by \cite[2.6]{OAT}. The result follows from Theorem \ref{E:0-dim}.
	\end{proof}
	
	\section{\bfseries{THE PROPERTY (*)}}
	
	In \cite{BR}, Brewer and Richman give some characterizations of zero-dimensional rings. Here we generalize some of these results under certain conditions. In particular, if $R$ is a principal ideal ring and $M$ is a finitely generated faithful multiplication $R$-module, the property (*) we introduced in Section 4 can be used to determine whether $M$ is zero-dimensional or not. Before giving that result we need some lemma. $R$ is assumed to be a principal ideal ring in the following.
	
	\begin{lemma}\label{E:pir}
		Let $M$ be a finitely generated multiplication $R$-module and $m$ an element of $M$ such that $Rm=IM$. The following conditions are equivalent:
		\begin{enumerate}[(i)]
			\item There exists an $n\in\nat$ such that $I^nM=I^{n+1}M$.
			\item $IM+\bigcup_{n=1}^\infty(0:_MI^n)=M$.
		\end{enumerate}
	\end{lemma}
	\begin{proof}
		(i)$\Rightarrow$(ii) Suppose that $I^nM=I^{n+1}M$ for some $n\in\nat$. Since $R$ is a principal ideal ring there exists an $r\in{R}$ such that $I=(r)$. Then $r^nM=r^{n+1}M$. That means for each $m\in{M}$ there exists a $m'\in{M}$ such that $r^nm=r^{n+1}m'$. Then $m-rm'\in(0:_Mr^n)=(0:_MI^n)$. Hence we have $M\subseteq{IM+\bigcup_{n=1}^\infty(0:_MI^n)}$. The result follows.
		
		(ii)$\Rightarrow$(i) Assume that $IM+\bigcup_{n=1}^\infty(0:_MI^n)=M$ holds. Since $I$ is a principal ideal, $I=(x)$ for some $x\in R$. Then we have $xM+\bigcup_{n=1}^\infty(0:_Mx^n)=M$. Let $\{ a_1,a_2,...,a_n\}$ be a generator set for $M$. Then, by assumption, we have $$a_i\in xM+\bigcup_{n=1}^\infty(0:_Mx^n)$$ for each $i\in \{1,2,...,n\}$. Then, for each $i\in \{1,2,...,n\}$,  there exists $m_i\in M$ and $n_i\in(0:x^{k_i})$, $k_i\in \mathbb{N}$ such that $a_i=xm_i+n_i$. Set $k=max\{k_1,...,k_n\}$. Then $$x^ka_i=x^{k+1}m_i+x^kn_i=x^{k+1}m_i\in x^{k+1}M.$$
		Hence we have $x^kM\subseteq x^{k+1}M$. The other inclusion is always true. Therefore we obtain $x^kM=x^{k+1}M$, that is $I^kM=I^{k+1}M$.
	\end{proof}
	\begin{theorem}\label{E:zerodimlemma}
		A finitely generated faithful multiplication $R$-module $M$ is zero-dimensional if and only if one of the conditions of Lemma \ref{E:pir}  are satisfied for every $m\in{M}$.
	\end{theorem}
	\begin{proof}
		Suppose that the condition (ii) of Lemma \ref{E:pir} is not satisfied for some $m\in{M}$. Then $IM+\bigcup_{n=1}^\infty(0:_MI^n)$ is a proper submodule of $M$, and hence it is contained in a prime submodule $Q'$ of $M$. Then the ideal $Q=(Q':M)$ is a prime ideal of $R$. Since $R$ is a principal ideal ring, we have $I=(a)$ for some $a\in{R}$. Set $S:=\{a^nr : n\in{\nat}, r\in{R\backslash{Q}}\}$. Assume $0\in{S}$. Then there exists an $r\in{R\backslash{Q}}$ such that $a^nr=0$. Let $x\in{M}$. Then we have $rx\in(0:_Ma^n)=(0:_MI^n)\subseteq{Q'}$. As $r\not\in{Q=(Q':M)}$ and $Q'$ is prime, we conclude that $x\in{Q'}$, and hence $M\subseteq{Q'}$, a contradiction. Hence $0\not\in{S}$. Then there exists a prime ideal $P$ of $R$ such that $P\cap{S}=\emptyset$. Since $R\backslash{Q}\subseteq{S}$ we have $P\subseteq{Q}$. Besides $aM\subseteq{Q'}$, and hence $a\in{Q}$. However, $a\not\in{P}$ since $a\in{S}$. Therefore $P$ is a proper ideal of $Q$. Then, by \cite[3.1]{BS}, $PM$ is a proper submodule of $QM$. Thus $M$ is not zero-dimensional.
		
		Conversely assume that $M$ is not zero-dimensional. Then there exist prime submodules $P$ and $Q$ of $M$ such that $P\subset{Q}$. Let $m\in{Q\backslash{P}}$. Then $Rm=IM$ for some ideal $I$ of $R$. Suppose $\bigcup_{n=1}^\infty(0:_MI^n)\not\subseteq{P}$. Then there exists an $x\in{M}$ such that $I^nx=0$ for some $n\in{\nat}$ and $x\not\in{P}$. Since $I=(b)$ for some $b\in{R}$, we have $b^nx=0\in{P}$. As $P$ is prime, we have $b^n\in(P:M)$. Then $b\in(P:M)$. Hence we obtain $m\in{Rm}=bM\subseteq{P}$, a contradiction. Therefore $\bigcup_{n=1}^\infty(0:_MI^n)\subseteq{P}\subset{Q}$. Besides $IM=Rm\subseteq{Q}$. Hence $IM+\bigcup_{n=1}^\infty(0:_MI^n)\subseteq{Q}\neq{M}$. This is the contrapositive of the condition (ii) of Lemma \ref{E:pir}.
	\end{proof}
	\begin{theorem}\label{E:zero-dim}
		Let $M$ be a finitely generated faithful multiplication $R$-module. The following conditions are equivalent:
		\begin{enumerate}[(i)]
			\item $M$ is zero-dimensional.
			\item Property (*) holds for the family of all submodules of $M$.
			\item Property (*) holds for the family of all primary submodules of $M$.
		\end{enumerate}
	\end{theorem}
	\begin{proof}
		(i)$\Rightarrow$(ii) Suppose that $M$ is zero-dimensional. Then by \ref{E:zerodimlemma}, for each $m\in{M}$ there exists an $n\in{\nat}$ such that $I^nM=I^{n+1}M$, where $Rm=IM$. Observe that the equality $I^{n+t}M=I^nM$ holds for all
		$t\in{\nat}$. Let $N$ be  a submodule of $M$. If $x\in{\rad{N}}$, then  there exists a $k\in{\nat}$ such that $J^kM\subseteq{N}$ where $Rx=JM$. If $n>k$ then $J^nM=J^{n-k}(J^kM)\subseteq{J^{n-k}N}\subseteq{N}$. If $n\le{k}$ then $J^nM=J^kM\subseteq{N}$. In both cases we have $x^n=J^nM\subseteq{N}$. Therefore the property (*) holds for the family of all submodules of $M$.
		
		(ii)$\Rightarrow$(iii) Trivial.
		
		(iii)$\Rightarrow$(i) Suppose that the property (*) holds for the family of all primary submodules of $M$ and $M$ is not zer-dimensional. Assume that $P$ is a prime submodule of $M$ that is not maximal. Let $x\in M$. Since $M$ is a multiplication module there is an ideal $I$ of $R$ such that $Rx=IM$. Let $P$ be a minimal prime submodule of $IM$. For each $n\in\mathbb{N}$, define
		$$Q_n=\{m\in M : ~sm\in I^nM \text{ for some } s\in R\backslash (P:M)\}.$$
		Set $Q=(P:M)$. Then, by \cite[6]{L}, $M_Q$ is a local module. Hence $P_Q$ is the unique maximal submodule of $M_Q$. Observe that $P$ is also a minimal prime submodule of $I^nM$. Then we have rad$(I^nM)_Q=P_Q$.  Thus,
		$$\text{rad}Q_n=\text{rad}((I^nM)_Q\cap R)=(\text{rad}(I^nM)_Q)\cap R=P_Q\cap R=P,$$
		and hence we obtain
		$$(P:M)M=P=\text{rad}Q_n=\text{rad}((Q_n:M)M)=\sqrt{(Q_n:M)}M.$$
		Since $M$ is a finitely generated faithful multiplication module, by \cite[3.1]{BS}, we conclude that $\sqrt{(Q_n:M)}=(P:M)$. Now, let $r\in R$, $m\in M$ such that $rm\in Q_n$. Then there exists $s\in R\backslash(P:M)$ such that $srm\in I^nM$.  If $r\not\in\sqrt{(Q_n:M)}$, then $sr\in R\backslash(P:M)$. Then since $srm\in I^nM$ we obtain $m\in Q_n$. Therefore $Q_n$ is $(P:M)$-primary.
		Then $\{Q_n\}_n\in \mathbb{N}$, is a family of primary submodules of $M$. Observe that $x\in P=\bigcap_{n\in\mathbb{N}}\text{rad}(Q_n)$. We are to show that $x\not\in\text{rad}(\bigcap_{n\in\mathbb{N}}Q_n)$. Assume, on the contrary, that $x\in\text{rad}(\bigcap_{n\in\mathbb{N}}Q_n)$. Then for some $k\in\mathbb{N}$ we have $x^k\subseteq \bigcap_{n\in\mathbb{N}}Q_n$. In particular, $x^k\subseteq Q_{k+1}$. Note that $I$ is a principal ideal, hence there exists $a\in R$ such that $I=(a)$. Since $P\neq M$ there exists $m\in M\backslash P$ and $s\in R\backslash(P:M)$ such that $sa^km=a^{k+1}m'$ for some $m'\in M$. Since $M$ is torsion-free, we conclude that $sm=am'\in IM\subseteq P$ and this contradicting our choice $m\in M\backslash P$. Hence we must have $x\not\in\text{rad}(\bigcap_{n\in\mathbb{N}}Q_n)$. Therefore we obtain a family of primary submodules $Q_n$, $n\in M$, of $M$ for which the property (*) does not hold.
	\end{proof}

\end{document}